\documentclass[11pt]{article}
\usepackage[a4paper,margin=28mm]{geometry}
\usepackage{amsmath,amssymb,amsthm,aliascnt,mathtools}
\usepackage{xcolor}
\usepackage{etoolbox}
\usepackage{hyperref}
\hypersetup{colorlinks=true,linkcolor=cyan,citecolor=red,urlcolor=black,filecolor=black}
\numberwithin{equation}{section}
\newtheorem{theorem}{Theorem}[section]
\newaliascnt{lemma}{theorem}
\newtheorem{lemma}[lemma]{Lemma}
\aliascntresetthe{lemma}

\newaliascnt{proposition}{theorem}
\newtheorem{proposition}[proposition]{Proposition}
\aliascntresetthe{proposition}

\newaliascnt{corollary}{theorem}
\newtheorem{corollary}[corollary]{Corollary}
\aliascntresetthe{corollary}

\DeclareMathOperator{\GF}{GF}
\DeclareMathOperator{\cl}{cl}
\DeclareMathOperator{\spanof}{span}
\newcommand{\BB}{\mathcal B}
\newcommand{\CC}{\mathcal C}
\newcommand{\FF}{\GF(2)}
\newcommand{\wloc}[1]{w_{#1}(M;B)}
\newcommand{\dloc}[1]{d_{#1}(M;B)}
\newcommand{\ind}{\mathbf 1}

\makeatletter
\patchcmd{\@afterheading}
  {\clubpenalty \@M}
  {\clubpenalty \@clubpenalty}
  {}
  {\PackageError{Manuscriot-6}{Unable to relax heading pagination}{}}
\makeatother

\begin{document}
\begin{center}
{\Large\bfseries A rank bound for bases and circuits in binary matroids}
\end{center}
\begin{center}
Houshan Fu\par

School of Mathematics and Information Science, Guangzhou University\\
Guangzhou 510006, Guangdong, P. R. China\par

Email address: fuhoushan@gzhu.edu.cn
\end{center}
\begin{abstract}
Let \(b(M)\), \(d(M)\), and \(r(M)\) denote the number of bases,
the number of circuits, and the rank of a matroid \(M\), respectively.
Using a weighted count of fundamental circuits, we prove that every
nonempty simple binary matroid \(M\) with no coloops satisfies
\[
 2b(M)\ge(r(M)+1)d(M),
\]
with equality if and only if \(M\) is isomorphic to the Fano matroid.
This settles a conjecture recorded by Oxley in 1983.
For every \(e\in E(M)\), we also establish the deletion bound
\(b(M\backslash e)\ge d(M)\) and characterize its equality cases.
The deletion bound yields a sharp linear lower bound for basis growth
under successive series extensions.
Finally, under a circuit-size exclusion condition, every pair \((C,e)\)
with \(C\) a circuit of a prescribed size and \(e\in C\) can be
assigned a distinct basis containing \(C-e\).
\end{abstract}

\section{Introduction}

For a matroid \(M\), let \(b(M)\), \(d(M)\), and \(r(M)\) denote
the number of bases, the number of circuits, and the rank, respectively.
For the cycle matroid of a connected graph, \(b(M)\) counts spanning
trees and \(d(M)\) counts cycles. Quirk and Seymour showed that a
connected simple graph has at least as many spanning trees as cycles,
a result recorded by Welsh as a private communication
\cite[p.~287]{Welsh1976}. For matroids of fixed rank and size, Welsh
also posed an average version of this comparison
\cite[p.~288, Exercise~5]{Welsh1976}.
However, for simple non-binary matroids, the ratio \(b(M)/d(M)\)
can be arbitrarily close to zero \cite[p.~363]{Oxley2011}.

For every simple binary matroid \(M\), Oxley
\cite[Theorem~2.1]{Oxley1983} established
\begin{equation}\label{eq:oxley}
 b(M)\ge2d(M),
\end{equation}
with equality precisely when deleting all coloops leaves the Fano
matroid \(F_7\). For nonempty simple binary matroids with no coloops,
Oxley further obtained the rank-dependent bound
\(b(M)>6(r(M)+1)d(M)/19\) \cite[Theorem~3.1]{Oxley1983}.
The referee of that paper conjectured the sharper inequality
\(2b(M)\ge(r(M)+1)d(M)\).
The conjecture is recorded as Conjectures~9.4.12 and~15.5.16
in \cite{Oxley2011}.
Our first theorem settles this conjecture and determines the equality case.

\begin{theorem}\label{thm:rank}
Let \(M\) be a nonempty simple binary matroid with no coloops. Then
\[
 2b(M)\ge(r(M)+1)d(M).
\]
Equality holds if and only if \(M\cong F_7\).
\end{theorem}

For ranks at least four, the rank bound strengthens \eqref{eq:oxley}. Equality at \(F_7\) shows that the coefficient \(2\) cannot be reduced. The exclusion of coloops is necessary because adjoining coloops
increases the rank without changing either count.

Oxley also compared the numbers of bases and circuits containing
a fixed independent set \cite[Theorem~1.3]{Oxley1983}.
We instead ask whether deleting a prescribed element leaves at least
as many bases as there are circuits of the original matroid.

\begin{theorem}\label{thm:deletion}
Let \(M\) be a nonempty simple binary matroid with no coloops, and \(e\in E(M)\). Then
\[
 b(M\backslash e)\ge d(M).
\]
Equality holds if and only if \(E(M)\) is a circuit, or \(M\) is a
parallel connection of a circuit and a triangle, with \(e\) outside the
triangle.
In particular, the inequality is strict
when \(M\) is cosimple.
\end{theorem}

The deletion bound also holds when coloops are allowed. If \(e\) is
a coloop, then \(b(M\backslash e)=b(M)>d(M)\).
If \(e\) is not a coloop, deleting all coloops reduces the bound
and its equality cases to
\autoref{thm:deletion}.

A one-element series extension preserves the number of circuits,
while the increase in the number of bases equals the number of bases
avoiding the extended element. Therefore, applying
\autoref{thm:deletion} at each step gives the following growth bound.

\begin{corollary}\label{cor:growth}
Let \(N\) be a nonempty simple binary matroid with no coloops, and
\(s\) be a nonnegative integer.
If \(M\) is obtained from \(N\) by \(s\) successive one-element
series extensions, then
\[
 b(M)\ge b(N)+s d(N).
\]
Equality holds if and only if \(s=0\), \(E(M)\) is a circuit, or
\(M\) is a parallel connection of a circuit and a triangle.
If \(N\) is cosimple and \(s\ge1\), the inequality is strict.
\end{corollary}

Finally, we ask when the independent sets \(C-e\), for circuits \(C\)
of prescribed sizes and \(e\in C\), can be extended simultaneously
to distinct bases.
Let \(M\) be a simple binary matroid and \(J\) a finite set of positive
integers. If \(M\) has no \((q-p+2)\)-element circuit whenever
\(p,q\in J\) and \(p<q\), \autoref{thm:distinct-extensions}
guarantees distinct bases containing \(C-e\) for the pairs \((C,e)\)
with \(C\) a circuit, \(|C|\in J\), and \(e\in C\).
Consequently, \(\sum_{j\in J}j d_j(M)\le b(M)\), where
\(d_j(M)\) denotes the number of \(j\)-element circuits.

The proofs use a weighted count in which each circuit contributes
its size, regardless of how many bases have it as a fundamental
circuit. Low-rank contraction counts and a folded-cube edge bound
control the three largest circuit sizes jointly, allowing deletion
induction to establish the rank bound. The same weights also verify
Hall's condition for distinct basis extensions.

The paper is organized as follows. \autoref{sec:prelim} develops
the weighted count. \autoref{sec:weights} derives the joint bounds,
and \autoref{sec:proofs} proves the rank and deletion bounds and
the growth corollary. \autoref{sec:extensions} constructs distinct
basis extensions.

\section{Preliminaries}\label{sec:prelim}

We collect the notation and basic matroid facts, then establish
counting identities for circuits and bases. These identities are used
both in the weighted bounds and in the construction of distinct
basis extensions.

\subsection{Notation and basic facts}

All matroids are finite. Terminology and notation follow
Oxley \cite{Oxley2011}.
We write \(E(M)\), \(\BB(M)\), \(\CC(M)\), and \(\cl_M\) for
the ground set, set of bases, set of circuits, and closure operator.
Deletion, contraction, restriction, and duality are denoted by
\(M\backslash X\), \(M/X\), \(M|X\), and \(M^*\), respectively.
We also write \(\CC^*(M)\coloneqq\CC(M^*)\) for the set of cocircuits.
For each integer \(j\ge1\), let \(d_j(M)\) be the number of \(j\)-element
circuits. For \(e\in E(M)\), let \(b_e(M)\) and \(d_e(M)\)
count the bases and circuits containing \(e\), respectively. We abbreviate
\(X-\{e\}\) to \(X-e\). A \emph{triangle} is a three-element circuit,
a \emph{triad} is a three-element cocircuit, and a \emph{series pair} is a
two-element cocircuit. A matroid is \emph{cosimple} if its dual is simple.
A \emph{parallel class} is a maximal set of nonloops any two
distinct members of which form a two-element circuit. For a direct sum \(M=M_1\oplus M_2\), we have
\begin{equation}\label{eq:direct-sum-counts}
 r(M)=r(M_1)+r(M_2),\qquad
 d(M)=d(M_1)+d(M_2),\qquad
 b(M)=b(M_1)b(M_2).
\end{equation}

For \(B\in\BB(M)\) and \(e\in E(M)-B\), the
\emph{fundamental circuit} \(C(e,B)\) is the unique circuit
contained in \(B\cup\{e\}\). For \(f\in B\), the
\emph{fundamental cocircuit} \(C^*(f,B)\) is the unique cocircuit
contained in \((E(M)-B)\cup\{f\}\).
Basis exchange gives
\begin{equation}\label{eq:exchange}
 (B-f)\cup\{e\}\in\BB(M)
 \quad\Longleftrightarrow\quad f\in C(e,B)
 \quad\Longleftrightarrow\quad e\in C^*(f,B).
\end{equation}
A matroid is \emph{binary} if it is representable over \(\FF\).
In a binary matroid \(M\), every circuit meets every cocircuit in an
even number of elements:
\begin{equation}\label{eq:binary-orthogonality}
 |C\cap C^*|\equiv0\pmod2,
 \qquad C\in\CC(M),\ C^*\in\CC^*(M).
\end{equation}
The symmetric difference of two circuits is a disjoint union
of circuits:
\begin{equation}\label{eq:binary-symmetric-difference}
 C_1\triangle C_2=D_1\mathbin{\dot\cup}\cdots\mathbin{\dot\cup}D_k,
 \qquad C_1,C_2,D_1,\ldots,D_k\in\CC(M).
\end{equation}
The union may be empty. By duality,
\eqref{eq:binary-symmetric-difference} also holds for cocircuits.
Contraction preserves cosimplicity, since deletion preserves
simplicity in the dual. If \(M\) is cosimple and \(e\in E(M)\), then
\(e\) is not a coloop and \(M^*/e\) has no loops, since \(M^*\) is simple.
Thus, \(r(M\backslash e)=r(M)\), and
\(M\backslash e=(M^*/e)^*\) has no coloops.

When using a representation of a simple binary matroid \(M\),
we fix one of full row rank, with columns labelled by \(E(M)\).
In vector sums, each element denotes its column in \(\FF^{r(M)}\).
These columns are distinct and nonzero. For every circuit,
\begin{equation}\label{eq:circuit-column-sum}
 \sum_{e\in C}e=0,\qquad C\in\CC(M).
\end{equation}
For a basis \(B\), put \(\sigma_B\coloneqq\sum_{x\in B}x\),
and use the columns of \(B\) as a coordinate basis.
For \(e\in E(M)-B\), its coordinates are the coefficients in its
unique expression as a linear combination of the columns of \(B\).
By \eqref{eq:circuit-column-sum}, the nonzero coordinates are
indexed precisely by the elements of \(C(e,B)-e\).
The notation \(\sigma_B\in E(M)\) means that this vector sum is itself
a column of the representation.
A circuit is \emph{spanning} if it has rank \(r(M)\), or
equivalently, size \(r(M)+1\).
If \(r(M)\ge2\), then
\begin{equation}\label{eq:spanning-bases}
 B\text{ lies in a spanning circuit}
 \quad\Longleftrightarrow\quad \sigma_B\in E(M),
\end{equation}
and that circuit is \(B\cup\{\sigma_B\}\).
Each spanning circuit contains \(r(M)+1\) bases,
and no basis belongs to two spanning circuits.
Consequently, for \(r(M)\ge2\),
\begin{equation}\label{eq:spanning-basis-count}
 \bigl|\{B\in\BB(M):\sigma_B\in E(M)\}\bigr|
 =(r(M)+1)d_{r(M)+1}(M).
\end{equation}
For every \(e\in E(F_7)\), we will also use the counts
\begin{equation}\label{eq:fano}
 b(F_7)=28,\quad d(F_7)=14,\quad
 b(F_7\backslash e)=16,\quad d(F_7\backslash e)=7.
\end{equation}

\subsection{Counting by bases and circuits}
To compare circuit counts with the number of bases, we count each
circuit through the bases at which it is fundamental.
This allows us to estimate the contribution at a fixed basis and
then sum over all bases. The same circuit may occur at several
bases. We first determine this multiplicity and choose weights
that give each circuit total weight equal to its size.

Consider pairs \((B,C)\) in which \(B\) is a basis of \(M\) and
\(C\) is a fundamental circuit with respect to \(B\).
Let \(X\subseteq E(M)\), and fix a basis \(B_X\) of \(M|X\).
The bases of \(M/X\) are given by
\begin{equation}\label{eq:contraction-bases}
 \BB(M/X)=\{B\subseteq E(M)-X:B\cup B_X\in\BB(M)\}.
\end{equation}
For a circuit \(C\) and an element \(e\in C\), the set \(C-e\)
is a basis of \(M|C\). By \eqref{eq:contraction-bases}, the map
\(B'\mapsto(C-e)\cup B'\) is a bijection from \(\BB(M/C)\) to
the bases of \(M\) containing \(C-e\).
Every such basis \(B\) avoids \(e\), since \(C\) is dependent,
and satisfies \(C(e,B)=C\).

For distinct \(e,f\in C\), no basis contains both \(C-e\) and
\(C-f\), since their union is \(C\).
Conversely, whenever \(C\) is fundamental with respect to a basis
\(B\), the unique element of \(C-B\) determines \(e\).
Counting over \(e\in C\) gives
\begin{equation}\label{eq:fundamental-count}
 \bigl|\{B\in\BB(M):|B\cap C|=|C|-1\}\bigr|=|C|b(M/C).
\end{equation}
To remove the extension factor \(b(M/C)\) from this count, assign
each pair \((B,C)\) weight \(1/b(M/C)\).
For a fixed basis \(B\), let \(\dloc{j}\) and \(\wloc{j}\)
be the number and total weight, respectively, of its
\(j\)-element fundamental circuits.
Explicitly,
\begin{align}
 \dloc{j}&\coloneqq\bigl|\{e\in E(M)-B:|C(e,B)|=j\}\bigr|,
 \notag\\
 \wloc{j}&\coloneqq\sum_{e\in E(M)-B,\,|C(e,B)|=j}
 \frac1{b(M/C(e,B))}.
 \label{eq:local-weight}
\end{align}
By \eqref{eq:fundamental-count}, each \(j\)-element circuit \(C\)
is fundamental at \(j b(M/C)\) bases. Its weight \(1/b(M/C)\)
at each basis gives total weight \(j\), so
\begin{equation}\label{eq:weight-total}
 \sum_{B\in\BB(M)}\wloc{j}=j d_j(M).
\end{equation}
The unweighted count follows directly from \eqref{eq:fundamental-count}:
\begin{equation}\label{eq:count-total}
 \sum_{B\in\BB(M)}\dloc{j}
 =j\sum_{C\in\CC(M),\,|C|=j}b(M/C).
\end{equation}
For a simple matroid, summing \eqref{eq:weight-total} after division
by \(j\) recovers the total circuit count:
\begin{equation*}
 d(M)=\sum_{B\in\BB(M)}\sum_{j=3}^{r(M)+1}\frac{\wloc{j}}j.
\end{equation*}

\section{Weighted circuit bounds}\label{sec:weights}
Oxley \cite[Lemmas~2.2 and~2.7]{Oxley1983} proved the separate bounds
\[
(r(M)+1)d_{r(M)+1}(M)\le b(M)\qquad\text{and}\qquad r(M)d_{r(M)}(M)\le b(M)
\]
for every simple binary matroid \(M\).
Oxley also used joint counts for two circuit sizes
\cite[proof of Lemma~2.11]{Oxley1983}.
To control the correction term in the deletion induction for
\autoref{thm:rank}, we need a bound on three circuit sizes together.
For every simple cosimple binary matroid \(M\) with \(r(M)\ge5\),
we establish
\[
 \frac12d_{r(M)-1}(M)+d_{r(M)}(M)+\frac32d_{r(M)+1}(M)
 <\frac{2b(M)}{r(M)+1}.
\]
By \eqref{eq:weight-total}, summing bounds on the local weights
yields bounds on circuit counts. Rank-one and rank-two contractions
control the local weights, while a folded-cube edge bound controls
the joint occurrence of circuits of sizes \(r(M)\) and \(r(M)+1\).

\subsection{Weights at a fixed basis}

To bound the weights at a basis, we first estimate the number of
bases in low-rank contractions.
For a circuit \(C\), we have \(r(M/C)=r(M)-|C|+1\)
by the contraction rank formula.
When \(|C|=r(M)\) or \(r(M)-1\), the contraction \(M/C\)
has rank one or two, respectively. If \(M\) is cosimple, so is
\(M/C\).
For the rank-two contractions, we use the following bounds, retaining any loops and
parallel elements of the contraction.

\begin{lemma}\label{lem:rank-two}
Let \(M\) be a cosimple binary matroid of rank two with \(t\)
nonloops.
\begin{enumerate}
\item[\(\mathrm{(i)}\)] We have \(b(M)\ge\max\{8,3(t-3)\}\).
\item[\(\mathrm{(ii)}\)] If \(M\) has three nonempty parallel classes, then
\(b(M)\ge3t-7\).
\item[\(\mathrm{(iii)}\)] If \(M\) has three nonempty parallel classes,
each of size at least two, then \(b(M)\ge4t-12\).
\end{enumerate}
\end{lemma}
\begin{proof}
The nonempty parallel classes correspond to the distinct nonzero
vectors occurring in a two-row binary representation.
Since \(\FF^2\) has three nonzero vectors and \(r(M)=2\),
there are two or three such classes.
A basis consists of two elements from distinct parallel classes.
The union of all but one of the parallel classes is a cocircuit.
Cosimplicity requires each such union to have at least three elements.

If there are two parallel classes of sizes \(a,b\), then \(a,b\ge3\) and \(t=a+b\ge 6\). Hence,
\[
 b(M)=ab=3(t-3)+(a-3)(b-3)\ge3(t-3)\ge9.
\]
If there are three parallel classes instead, order their sizes as \(1\le a\le b\le c\). The cocircuit condition implies \(a+b\ge3\), which forces \(b,c\ge2\). With \(t=a+b+c\ge 5\) and \(b(M)=ab+ac+bc\), we have
\[
 b(M)-3t+7=(a-1)(b+c-3)+(b-2)(c-2)\ge0,
\]
proving \(\mathrm{(ii)}\). Since \(t\ge5\), we also have \(b(M)\ge3t-7\ge\max\{8,3(t-3)\}\), completing \(\mathrm{(i)}\). Finally, if \(a,b,c\ge2\), then
\[
 b(M)-4t+12=(a-2)(b-2)+(a-2)(c-2)+(b-2)(c-2)\ge0,
\]
which proves \(\mathrm{(iii)}\).
\end{proof}

For the local counts below, fix a simple binary matroid \(M\)
with \(r(M)\ge4\) and a basis \(B\) of \(M\).
To bound \(b(M/C)\) for circuits of sizes \(r(M)-1\) and \(r(M)\),
we describe each fundamental circuit by the elements it omits from \(B\).

Each fundamental circuit contains exactly one element outside \(B\).
An \(r(M)\)-element fundamental circuit therefore contains
\(r(M)-1\) elements of \(B\) and omits exactly one element \(x\).
By \eqref{eq:circuit-column-sum}, its element outside \(B\) is
\(\sigma_B+x\). Conversely, whenever \(\sigma_B+x\in E(M)\),
its nonzero \(B\)-coordinates are indexed by \(B-x\). Since
\(r(M)\ge4\), this element lies outside \(B\) and determines the
fundamental circuit \((B-x)\cup\{\sigma_B+x\}\).
Thus, the set
\[
 X_B\coloneqq\{x\in B:\sigma_B+x\in E(M)\}
\]
indexes the \(r(M)\)-element fundamental circuits, with
\(x\in X_B\) corresponding to
\[
 C_x\coloneqq(B-x)\cup\{\sigma_B+x\}.
\]
Similarly, an \((r(M)-1)\)-element fundamental circuit contains
\(r(M)-2\) elements of \(B\) and omits a pair \(\{x,y\}\).
Its element outside \(B\) is \(\sigma_B+x+y\), whose nonzero
\(B\)-coordinates are indexed by \(B-\{x,y\}\). Conversely,
any such column in \(E(M)\) lies outside \(B\) and determines
the corresponding fundamental circuit.
Let \(G_B\) be the simple graph on \(B\) in which distinct \(x,y\) are adjacent
if and only if \(\sigma_B+x+y\in E(M)\). For an edge \(xy\), the circuit is
\[
 C_{xy}\coloneqq(B-\{x,y\})\cup\{\sigma_B+x+y\}.
\]
Consequently,
\begin{equation}\label{eq:graph-counts}
 |X_B|=\dloc{r(M)},\qquad |E(G_B)|=\dloc{r(M)-1}.
\end{equation}

Write \(\delta_B(x)\) for the degree of \(x\) in \(G_B\).
Thus, \(\delta_B(x)\) counts the \((r(M)-1)\)-element fundamental circuits omitting \(x\).
For \(x\in X_B\), the set \(B-x\) is a basis of \(M|C_x\).
Contraction by \(C_x\) quotients the representation by the span of
\(B-x\) and removes the elements of \(C_x\). The resulting rank-one
representation retains only the \(x\)-coordinate of each column outside \(C_x\).
Every element with \(x\)-coordinate \(1\) is a nonloop of \(M/C_x\)
and hence forms a basis on its own.

To count these nonloops, select \(x\), the elements \(\sigma_B+z\) with \(z\in X_B-x\),
and the elements \(\sigma_B+e+f\) with \(ef\in E(G_B)\) and
\(x\notin\{e,f\}\). Each selected column has \(x\)-coordinate \(1\).
The columns in the three families have, respectively, \(1\), \(r(M)-1\), and
\(r(M)-2\) nonzero \(B\)-coordinates. Since \(r(M)\ge4\), the
families are disjoint. Their elements are distinct within each family,
and none belongs to \(C_x\).
The selected elements form distinct singleton bases of \(M/C_x\),
even when their columns in the contraction coincide. The three families contain
\(1\), \(|X_B|-1\), and \(|E(G_B)|-\delta_B(x)\) elements,
respectively. These counts, together with \eqref{eq:graph-counts} and
\(\delta_B(x)\le |E(G_B)|\), imply
\begin{align}\label{eq:rank-one-selected}
 b(M/C_x)&\ge 1+(|X_B|-1)+(|E(G_B)|-\delta_B(x))\notag\\
 &=\dloc{r(M)-1}+\dloc{r(M)}-\delta_B(x)
 \ge\dloc{r(M)}.
\end{align}
If \(X_B\ne\varnothing\), \eqref{eq:local-weight},
\eqref{eq:graph-counts}, and \eqref{eq:rank-one-selected} give
\begin{equation}\label{eq:general-rank-one-weight}
 \wloc{r(M)}=\sum_{x\in X_B}\frac1{b(M/C_x)}
 \le\sum_{x\in X_B}\frac1{|X_B|}=1.
\end{equation}
The same bound holds when \(X_B=\varnothing\), since the sum is zero.

Cosimplicity provides a second lower bound.
For \(x\in X_B\), the closure of \(C_x\) is a hyperplane,
since \(r(M|C_x)=r(M)-1\).
Its complement \(E(M)-\cl_M(C_x)\) is a cocircuit and consists
precisely of the nonloops of \(M/C_x\).
If \(M\) is cosimple, every cocircuit has at least three elements, so
\[
 b(M/C_x)=|E(M)-\cl_M(C_x)|\ge3.
\]

For the \((r(M)-1)\)-element circuits, we need both a nonloop count
and information on the parallel classes of the rank-two contraction.
Fix \(xy\in E(G_B)\).
The set \(B-\{x,y\}\) is a basis
of \(M|C_{xy}\), so \(\{x,y\}\) is a basis of \(M/C_{xy}\)
by \eqref{eq:contraction-bases}.
Use a representation with column \((1,0)\) for \(x\)
and column \((0,1)\) for \(y\). Since \(C_{xy}\) is spanned
by \(B-\{x,y\}\), contraction retains just the \(x\)- and
\(y\)-coordinates of each column outside \(C_{xy}\).
Select \(x,y\), the columns \(\sigma_B+z\) with \(z\in X_B\),
and the columns \(\sigma_B+e+f\) with \(ef\in E(G_B)-xy\).
As in the rank-one case, the selected elements are distinct.
Each has a nonzero \(x\)- or \(y\)-coordinate, giving
\(\dloc{r(M)-1}+\dloc{r(M)}+1\) nonloops of \(M/C_{xy}\).
Write \(\ind_{X_B}\) for the indicator of \(X_B\).
The selected elements fall into the following parallel classes.
The counts need not include every element of a class.
\begin{equation}\label{eq:selected-classes}
\begin{array}{c|l}
 \text{Column in }M/C_{xy} & \text{Number of selected elements}\\
 \hline
 (0,1) & \delta_B(x)+\ind_{X_B}(x)\\
 (1,0) & \delta_B(y)+\ind_{X_B}(y)\\
 (1,1) & |X_B-\{x,y\}|+|E(G_B[B-\{x,y\}])|
\end{array}
\end{equation}
For example, the selected elements with column \((0,1)\) are \(y\),
the elements \(\sigma_B+x+z\) with \(xz\in E(G_B)\) and \(z\ne y\),
and \(\sigma_B+x\) when \(x\in X_B\).
The elements in the third row are \(\sigma_B+z\) for
\(z\in X_B-\{x,y\}\) and \(\sigma_B+e+f\) for edges \(ef\)
disjoint from \(xy\).
Since \(xy\in E(G_B)\), the first two numbers are positive.
The third is at least \(|X_B|-2\), so all three are positive if
\(\dloc{r(M)}\ge3\). If \(X_B=B\), the first two are at least two,
and the third is at least \(r(M)-2\ge2\).
Elements with the same nonzero column are parallel, and any two
selected elements from different parallel classes form a basis
of \(M/C_{xy}\).
Any additional nonloops can only increase the number of bases.

For cosimple \(M\), the preceding counts bound the two weights
separately. Keeping their common dependence on \(\dloc{r(M)}\)
also yields the joint bound in part \(\mathrm{(iv)}\) below.
\begin{lemma}\label{lem:local}
For every basis \(B\) of a simple cosimple binary matroid \(M\)
with \(r(M)\ge4\), the following hold.
\begin{enumerate}
\item[\(\mathrm{(i)}\)] \(\wloc{r(M)}\le\min\{1,\dloc{r(M)}/3\}\).
\item[\(\mathrm{(ii)}\)]\(\wloc{r(M)-1}\le\max\{1/3,7/12-\dloc{r(M)}/8\}\).
\item[\(\mathrm{(iii)}\)] If \(\dloc{r(M)}>0\), then
\[
 \wloc{r(M)}\le
 \frac{\dloc{r(M)}^2+\sum_{x\in X_B}\delta_B(x)}
 {\dloc{r(M)}(\dloc{r(M)-1}+\dloc{r(M)})},
\]
where
\[
 \sum_{x\in X_B}\delta_B(x)
 \le\min\Bigl\{2\dloc{r(M)-1},\,
 \dloc{r(M)-1}+\binom{\dloc{r(M)}}2\Bigr\}.
\]
\item[\(\mathrm{(iv)}\)] If \(\dloc{r(M)}\ge2\), then
\[
 \wloc{r(M)}+3\wloc{r(M)-1}
 \le1+\frac2{\dloc{r(M)}}.
\]
\end{enumerate}
\end{lemma}
\begin{proof}
For \(\mathrm{(i)}\), each \(x\in X_B\) contributes
\(1/b(M/C_x)\) to \(\wloc{r(M)}\).
By \eqref{eq:rank-one-selected} and cosimplicity,
\(b(M/C_x)\ge\max\{3,\dloc{r(M)}\}\). Hence,
\begin{align*}
 \wloc{r(M)}&=\sum_{x\in X_B}\frac1{b(M/C_x)}
 \le\sum_{x\in X_B}\frac1{\max\{3,\dloc{r(M)}\}}\\
 &=\frac{\dloc{r(M)}}{\max\{3,\dloc{r(M)}\}}
 =\min\{1,\dloc{r(M)}/3\}.
\end{align*}

For each \(xy\in E(G_B)\), the contraction \(M/C_{xy}\)
is cosimple and binary of rank two. By \eqref{eq:selected-classes},
it has at least \(\dloc{r(M)-1}+\dloc{r(M)}+1\) nonloops.
Applying \autoref{lem:rank-two}\(\mathrm{(i)}\) to these contractions
bounds the sum in \eqref{eq:local-weight} as follows:
\begin{equation}\label{eq:rank-two-selected}
 \wloc{r(M)-1}=\sum_{xy\in E(G_B)}\frac1{b(M/C_{xy})}\le\frac{\dloc{r(M)-1}}{\max\{8,3(\dloc{r(M)-1}+\dloc{r(M)}-2)\}}.
\end{equation}
To remove the dependence on \(\dloc{r(M)-1}\), write
the numerator in \eqref{eq:rank-two-selected} as
\[
 \dloc{r(M)-1}
 =(\dloc{r(M)-1}+\dloc{r(M)}-2)+(2-\dloc{r(M)}).
\]
After division by the denominator in \eqref{eq:rank-two-selected},
the two terms contribute at most \(1/3\) and
\(\max\{0,2-\dloc{r(M)}\}/8\), respectively. Thus,
\[
 \wloc{r(M)-1}\le\frac13+\frac{\max\{0,2-\dloc{r(M)}\}}8
 =\max\Bigl\{\frac13,\frac7{12}-\frac{\dloc{r(M)}}8\Bigr\},
\]
which proves \(\mathrm{(ii)}\).

For \(\mathrm{(iii)}\), suppose \(\dloc{r(M)}>0\).
To retain the dependence on the graph degrees, use
\eqref{eq:rank-one-selected} in \eqref{eq:local-weight}:
\begin{align*}
 (\dloc{r(M)-1}+\dloc{r(M)})\wloc{r(M)}&=\sum_{x\in X_B}\frac{\dloc{r(M)-1}+\dloc{r(M)}}{b(M/C_x)}\\
 &\le\sum_{x\in X_B}\Bigl(1+\frac{\delta_B(x)}{b(M/C_x)}\Bigr)\\
 &\le\dloc{r(M)}+\frac{\sum_{x\in X_B}\delta_B(x)}{\dloc{r(M)}}.
\end{align*}
The last step uses \(b(M/C_x)\ge\dloc{r(M)}\) from
\eqref{eq:rank-one-selected}.
Here \(|X_B|=\dloc{r(M)}\) by \eqref{eq:graph-counts}.
Dividing by \(\dloc{r(M)-1}+\dloc{r(M)}\) proves the first inequality in
\(\mathrm{(iii)}\).
Each edge of \(G_B\) contributes at most two to the degree sum.
Therefore, \(\sum_{x\in X_B}\delta_B(x)\le2|E(G_B)|=2\dloc{r(M)-1}\).
For the other bound, count every edge once and each edge with both
ends in \(X_B\) once more. At most
\(\binom{|X_B|}2=\binom{\dloc{r(M)}}2\) edges receive this extra count.

Finally, for \(\mathrm{(iv)}\), the case \(\dloc{r(M)-1}=0\)
follows from \(\mathrm{(i)}\). Otherwise, the denominators below
are positive, since \(\dloc{r(M)-1}>0\) and \(\dloc{r(M)}\ge2\).
To add the bounds from \(\mathrm{(iii)}\) and
\eqref{eq:rank-two-selected}, we put both in a form with the common
factor \(\dloc{r(M)-1}+\dloc{r(M)}\) in the denominator:
\[
 \wloc{r(M)}\le\frac{\dloc{r(M)}^2+2\dloc{r(M)-1}}
 {\dloc{r(M)}(\dloc{r(M)-1}+\dloc{r(M)})},
\]
\[
 3\wloc{r(M)-1}\le\frac{\dloc{r(M)-1}}{\dloc{r(M)-1}+\dloc{r(M)}-2}\le\frac{\dloc{r(M)-1}+2}{\dloc{r(M)-1}+\dloc{r(M)}}.
\]
Clearing the positive denominators reduces the last inequality to
\(2(\dloc{r(M)}-2)\ge0\).
The two upper bounds sum to \(1+2/\dloc{r(M)}\), proving \(\mathrm{(iv)}\).
\end{proof}

\subsection{A folded-cube estimate}

We next relate circuits of size \(r(M)\) to spanning circuits
by counting the bases at which both are fundamental.
For a fixed \(r(M)\)-element circuit, these bases correspond to
edges of a folded cube, with two bases for each edge.
We establish this correspondence before estimating the edge count.

For an integer \(\ell\ge4\), the \emph{folded \(\ell\)-cube}, denoted by \(\square_\ell\), is
obtained from the ordinary \((\ell-1)\)-cube \(Q_{\ell-1}\) by adding an edge
between each pair of antipodal vertices. Its vertices are the
vectors in \(\FF^{\ell-1}\): edges of the ordinary \((\ell-1)\)-cube join
vectors differing in exactly one coordinate, and antipodal edges join
vectors differing in every coordinate.
For a graph \(G\) and \(U\subseteq V(G)\), write
\(\varepsilon_G(U)\coloneqq|E(G[U])|\).

For a simple binary matroid \(M\) with \(r(M)\ge4\), fix an
\(r(M)\)-element circuit \(C\) and set \(U_C:=E(M)-\cl_M(C)\).
Since \(r(M|C)=r(M)-1\), the space \(\spanof(C)\) is a hyperplane
in \(\FF^{r(M)}\) and \(U_C\ne\varnothing\).
To obtain cube coordinates on \(U_C\), choose \(e_0\in U_C\)
and write \(C=\{e_1,\ldots,e_{r(M)}\}\).
The columns \(e_1,\ldots,e_{r(M)-1}\) form a basis of \(\spanof(C)\),
and adjoining \(e_0\) gives a basis of \(\FF^{r(M)}\).
Every vector \(v\in\FF^{r(M)}\) has a unique expression
\[
 v=\epsilon_0e_0+\sum_{i=1}^{r(M)-1}\epsilon_i e_i,
 \qquad \epsilon_i\in\FF,\quad 0\le i\le r(M)-1.
\]
In this expression, \(v\in\spanof(C)\) if and only if \(\epsilon_0=0\).
Since \(\FF=\{0,1\}\), every vector outside \(\spanof(C)\) has
\(\epsilon_0=1\). Consequently, the expansion defines a bijection
\(\FF^{r(M)}-\spanof(C)\to\FF^{r(M)-1}\) by
\[
 e_0+\sum_{i=1}^{r(M)-1}\epsilon_i e_i \mapsto (\epsilon_1,\ldots,\epsilon_{r(M)-1}).
\]
On \(\FF^{r(M)}-\spanof(C)\), join two distinct vectors when their sum belongs to \(C\).
In the displayed coordinates, adding \(e_i\), for \(1\le i\le r(M)-1\),
changes only the \(i\)th coordinate.
By \eqref{eq:circuit-column-sum},
\(e_{r(M)}=\sum_{i=1}^{r(M)-1}e_i\), so adding \(e_{r(M)}\) changes every coordinate.
Thus, the above bijection is a graph isomorphism onto \(\square_{r(M)}\),
under which we regard \(U_C\) as a vertex subset of \(\square_{r(M)}\).

\begin{lemma}\label{lem:cube-correspondence}
Let \(M\) be a simple binary matroid with \(r(M)\ge4\) and
\(C\) an \(r(M)\)-element circuit. Then
\[
 \bigl|\{B\in\BB(M):|B\cap C|=r(M)-1,\ \sigma_B\in E(M)\}\bigr|
 =2\varepsilon_{\square_{r(M)}}(U_C).
\]
\end{lemma}
\begin{proof}
We associate each basis counted on the left with an edge of
\(\square_{r(M)}[U_C]\) and show that every edge arises twice.
For such a basis \(B\), write \(C-B=\{z\}\) and \(B-C=\{e\}\). Then the set \(C-z\) is a basis of \(\spanof(C)\).
Since independence of \(B\) forces \(e\notin\spanof(C)\),  \(e\in U_C\).
By \eqref{eq:circuit-column-sum},
\[
 \sigma_B=e+\sum_{x\in C-z}x=e+z.
\]
If \(\sigma_B\in\spanof(C)\), then \(e=\sigma_B+z\in\spanof(C)\),
a contradiction. Therefore, \(\sigma_B\in U_C\), as
\(\sigma_B\in E(M)\) by hypothesis. Moreover, \(z\ne0\) by simplicity,
so the relation \(e+\sigma_B=z\in C\) shows that \(e\) and
\(\sigma_B\) are distinct and adjacent.

Conversely, fix an edge \(\{e,f\}\) and put \(z=e+f\in C\).
Any basis producing this edge omits \(z\) and has \(e\) or \(f\) as its
unique element outside \(C\), so it is one of
\[
 (C-z)\cup\{e\},\qquad (C-z)\cup\{f\}.
\]
Both are bases because \(C-z\) is a basis of \(\spanof(C)\) and
\(e,f\notin\spanof(C)\).
Their column sums are \(f\) and \(e\), respectively, so both are counted
on the left. Since \(e\ne f\) and \(e,f\notin C\), the two bases are distinct.
Therefore, each edge corresponds to exactly two bases.
\end{proof}

We estimate this edge count using an initial segment of the ordinary cube.
For each integer \(t\ge0\), let \(F(t)\) count the edges induced by the first \(t\)
vertices of an ordinary cube with at least \(t\) vertices, ordered
lexicographically by their coordinates, with \(0<1\).
Coordinate reversal recovers the binary ordering of Ellis and Leader without
changing adjacency. Increasing the ambient dimension only adds initial zero
coordinates, so \(F(t)\) is independent of the dimension. Also, \(F(0)=0\).
For \(U\subseteq V(\square_\ell)=\FF^{\ell-1}\), the bound
\(|U|\le2^{\ell-1}\) allows \(F(|U|)\) to be computed in the
ordinary \((\ell-1)\)-cube.
Let \(\overline U\) be the set of antipodes of vertices in \(U\).
Each antipodal edge of \(\square_\ell[U]\) contributes two
vertices to \(U\cap\overline U\). Hence,
\[
 2\varepsilon_{\square_\ell}(U)=2\varepsilon_{Q_{\ell-1}}(U)+|U\cap\overline U|.
\]
In dimension \(\ell-1\), Theorem~3 of \cite{EllisLeader2018} bounds
\(2\varepsilon_{Q_{\ell-1}}(U)+|U\cap\overline U|\) when
\(|U|\le2^{\ell-2}\), and equation~(7) of the same paper gives the bound
when \(|U|\ge2^{\ell-2}\). Both results apply to arbitrary subsets \(U\).
After division by two, the two bounds take the form
\begin{equation}\label{eq:cube-initial}
 \varepsilon_{\square_\ell}(U)
 \le F(|U|)+\max\{0,|U|-2^{\ell-2}\}.
\end{equation}
For the weighted circuit count, we need an explicit bound in \(|U|\)
rather than the initial-segment function \(F\).
The next lemma converts \eqref{eq:cube-initial} into the required quadratic bounds.

\begin{lemma}\label{lem:cube}
For \(\ell\ge4\) and \(U\subseteq V(\square_\ell)\), we have
\[
 \varepsilon_{\square_\ell}(U)\le
 \begin{cases}
 |U|^2/4,&\ell=4,\\
 |U|/2+|U|^2/8,&\ell=5,\\
 |U|+|U|^2/16,&\ell\ge6.
 \end{cases}
\]
\end{lemma}
\begin{proof}
We first bound \(F(t)\) by the three quadratic expressions in the statement.
Write these as \(G_k(t)\coloneqq(k-1)t/2+t^2/2^{k+1}\),
where \(k\in\{1,2,3\}\).
The parameter \(k\) is also the dimension of the full subcubes used below.
We prove by induction on \(k\) that
\begin{equation}\label{eq:ordinary-cube}
 F(t)\le G_k(t),\qquad t\in\{0,\ldots,2^{\ell-1}\}.
\end{equation}
For \(k=1\), the subgraph defining \(F(t)\) is bipartite.
Hence, \(F(t)\le t^2/4=G_1(t)\).

Let \(k\in\{2,3\}\), and assume that \eqref{eq:ordinary-cube} holds
with \(k-1\) in place of \(k\) for every integer \(0\le t\le2^{\ell-1}\).
Write \(t=2^k a+s\), where \(a,s\) are integers with \(a\ge0\) and \(0\le s<2^k\).
Since \(\ell-1\ge k\), partitioning the first \(t\) vertices by their
first \(\ell-1-k\) coordinates yields \(a\) full ordinary \(k\)-cubes
and a possibly empty remaining set of \(s\) vertices.
When \(s>0\), deleting the common prefix identifies the remaining set
with the first \(s\) vertices of an ordinary \(k\)-cube in lexicographic order.
The remaining set induces \(F(s)\) edges, with \(F(0)=0\).
Each ordinary \(k\)-cube has \(k2^{k-1}\) edges.
An edge joining distinct sets leaves the last \(k\) coordinates unchanged,
so each vertex has at most one neighbour in any other set.
Thus, at most \(2^k\) edges join any two of these ordinary \(k\)-cubes,
and at most \(s\) edges join each ordinary \(k\)-cube to the remaining set.
Counting edges inside full cubes, between two full cubes, from full cubes
to the remaining set, and inside that set gives
\[
 F(t)\le ak2^{k-1}+2^k\binom a2+as+F(s).
\]
The induction hypothesis applies because \(s<2^k\le2^{\ell-1}\),
so \(F(s)\le G_{k-1}(s)\). With \(t=2^ka+s\),
the preceding estimate simplifies to
\[
 F(t)\le G_k(t)-\frac{s(2^k-s)}{2^{k+1}}\le G_k(t).
\]
This proves \eqref{eq:ordinary-cube}.

Choose \(k=1\) if \(\ell=4\), \(k=2\) if \(\ell=5\), and
\(k=3\) otherwise. Then \(\ell\ge k+3\), and the required bound is
\(\varepsilon_{\square_\ell}(U)\le G_k(|U|)\).
For \(|U|\le2^{\ell-2}\), the correction term in
\eqref{eq:cube-initial} vanishes, so \eqref{eq:ordinary-cube} gives the bound.
However, when \(|U|>2^{\ell-2}\), we must also control the positive
correction term. Write
\(|U|=2^{\ell-2}+s\), where \(1\le s\le2^{\ell-2}\).
In the ordinary \((\ell-1)\)-cube, the first \(2^{\ell-2}\) vertices
are those with first coordinate \(0\) and induce an ordinary \((\ell-2)\)-cube.
The next \(s\) vertices have first coordinate \(1\), induce \(F(s)\) edges,
and each has exactly one neighbour in the first subcube. Hence,
\(F(2^{\ell-2}+s)=F(2^{\ell-2})+s+F(s)\).
Combining \eqref{eq:cube-initial} and \eqref{eq:ordinary-cube}
with \(\ell\ge k+3\) absorbs the correction term:
\[
 \varepsilon_{\square_\ell}(U)\le G_k(2^{\ell-2})+2s+G_k(s)=G_k(|U|)-\bigl(2^{\ell-k-2}-2\bigr)s\le G_k(|U|).\qedhere
\]
\end{proof}

We now combine \autoref{lem:cube-correspondence} with the edge bound
and the weighted count by bases to obtain joint estimates for
circuits of sizes \(r(M)\) and \(r(M)+1\).

\begin{lemma}\label{lem:cube-count}
Let \(M\) be a simple binary matroid with \(r(M)\ge4\).
Then the following hold.
\begin{enumerate}
\item[\(\mathrm{(i)}\)] If \(r(M)=4\), then
\[
 4d_4(M)+5d_5(M)
 \le b(M)+\frac18\sum_{B\in\BB(M)}\dloc4.
\]
\item[\(\mathrm{(ii)}\)] If \(r(M)=5\), then
\[
 4d_5(M)+6d_6(M)
 \le b(M)+\frac1{20}\sum_{B\in\BB(M)}\dloc5.
\]
\item[\(\mathrm{(iii)}\)] If \(r(M)\ge6\), then
\[
 (r(M)-2)d_{r(M)}(M)+(r(M)+1)d_{r(M)+1}(M)
 \le b(M)+\frac1{8r(M)}\sum_{B\in\BB(M)}\dloc{r(M)}.
\]
\end{enumerate}
\end{lemma}
\begin{proof}
We split the total weight according to whether \(\sigma_B\in E(M)\).
For bases outside spanning circuits, the bound \(\wloc{r(M)}\le1\)
from \eqref{eq:general-rank-one-weight} suffices.
By \eqref{eq:spanning-basis-count},
\begin{equation}\label{eq:basis}
 \sum_{B\in\BB(M),\,\sigma_B\notin E(M)}\wloc{r(M)}
 \le b(M)-(r(M)+1)d_{r(M)+1}(M).
\end{equation}
For each \(r(M)\)-element circuit \(C\), put
\(U_C=E(M)-\cl_M(C)\), as in \autoref{lem:cube-correspondence}.
The nonloops of the rank-one contraction \(M/C\) are precisely
the elements of \(U_C\). Hence,
\begin{equation}\label{eq:outside-hyperplane}
 |U_C|=b(M/C).
\end{equation}
For the remaining bases, \autoref{lem:cube-correspondence}
counts the occurrences of each \(r(M)\)-element circuit \(C\)
as a fundamental circuit.
Each such occurrence has weight \(1/b(M/C)\) by \eqref{eq:local-weight}.
Summing over all \(r(M)\)-element circuits gives
\begin{equation}\label{eq:cube-weight-count}
 \sum_{B\in\BB(M),\,\sigma_B\in E(M)}\wloc{r(M)}
 =2\sum_{C\in\CC(M),\,|C|=r(M)}
 \frac{\varepsilon_{\square_{r(M)}}(U_C)}{b(M/C)}.
\end{equation}
By \eqref{eq:weight-total}, we have
\[
 r(M)d_{r(M)}(M)
 =\sum_{B\in\BB(M),\,\sigma_B\notin E(M)}\wloc{r(M)}
 +\sum_{B\in\BB(M),\,\sigma_B\in E(M)}\wloc{r(M)}.
\]
Substitute \eqref{eq:basis} and \eqref{eq:cube-weight-count} into this
decomposition and move the spanning-circuit term to the left:
\begin{equation}\label{eq:cube-master}
 r(M)d_{r(M)}(M)+(r(M)+1)d_{r(M)+1}(M)
 \le b(M)+2\sum_{C\in\CC(M),\,|C|=r(M)}
 \frac{\varepsilon_{\square_{r(M)}}(U_C)}{b(M/C)}.
\end{equation}
By \eqref{eq:outside-hyperplane}, dividing the quadratic edge bounds
in \autoref{lem:cube}, with \(\ell=r(M)\), by \(b(M/C)=|U_C|\)
leaves only constant and linear terms in \(b(M/C)\):
\[
 \frac{2\varepsilon_{\square_{r(M)}}(U_C)}{b(M/C)}
 \le\begin{cases}
 b(M/C)/2,&r(M)=4,\\
 1+b(M/C)/4,&r(M)=5,\\
 2+b(M/C)/8,&r(M)\ge6.
 \end{cases}
\]
The linear terms can now be summed by \eqref{eq:count-total}:
\[
 \sum_{C\in\CC(M),\,|C|=r(M)}b(M/C)
 =\frac1{r(M)}\sum_{B\in\BB(M)}\dloc{r(M)}.
\]
Substituting this identity and the preceding bounds in
\eqref{eq:cube-master} proves the three assertions.
\end{proof}

\subsection{The joint bounds}

We now use \eqref{eq:weight-total} to combine the circuit counts
with the sums over bases in \autoref{lem:cube-count}.
The resulting contribution at each basis is bounded by
\autoref{lem:local}.
In rank four, every circuit has size three, four, or five,
so these estimates bound \(d(M)\) directly.

\begin{proposition}\label{prop:rank-four}
Let \(M\) be a simple cosimple binary matroid of rank four. Then
\[
 d(M)<\frac25 b(M).
\]
\end{proposition}
\begin{proof}
Since \(d(M)=d_3(M)+d_4(M)+d_5(M)\),
\autoref{lem:cube-count}\(\mathrm{(i)}\) and \eqref{eq:weight-total} give
\begin{align}
 d(M)&\le\frac15b(M)+d_3(M)+\frac15d_4(M)
       +\frac1{40}\sum_{B\in\BB(M)}\dloc4\notag\\
 &=\frac15b(M)+\sum_{B\in\BB(M)}
 \Bigl(\frac{\wloc3}3+\frac{\wloc4}{20}+\frac{\dloc4}{40}\Bigr).
 \label{eq:rank-four-target}
\end{align}
To obtain \(d(M)<2b(M)/5\) from \eqref{eq:rank-four-target},
it suffices to show, for each \(B\in\BB(M)\), that
\begin{equation}\label{eq:rank-four-local-target}
 20\wloc3+3\wloc4+\frac32\dloc4<12.
\end{equation}
By \eqref{eq:graph-counts},
\[
 0\le\dloc3\le6,\qquad 0\le\dloc4\le4.
\]
If \(\dloc4\le2\), \autoref{lem:local}\(\mathrm{(i),(ii)}\) imply
\(\wloc4\le\dloc4/3\) and
\(\wloc3\le7/12-\dloc4/8\). Hence,
\[
 20\wloc3+3\wloc4+\frac32\dloc4
 \le20\Bigl(\wloc3+\frac{\dloc4}8\Bigr)\le\frac{35}3<12.
\]
For \(\dloc4=3\), we use the parallel-class counts in
\eqref{eq:selected-classes} to sharpen the estimates.
All three multiplicities are positive.
Each \(M/C_{xy}\), with \(xy\in E(G_B)\), has at least
\(\dloc3+4\) nonloops.
Combining the degree-sum bound \(\sum_{x\in X_B}\delta_B(x)\le\dloc3+3\)
in \autoref{lem:local}\(\mathrm{(iii)}\) with
\autoref{lem:rank-two}\(\mathrm{(ii)}\) bounds the two weights:
\[
 \wloc3\le\frac{\dloc3}{3\dloc3+5},\qquad
 \wloc4\le\frac13+\frac3{\dloc3+3}.
\]
Consequently, the left-hand side of \eqref{eq:rank-four-local-target}
is at most
\[
 \frac{11}2+\frac{20\dloc3}{3\dloc3+5}+\frac9{\dloc3+3}
 =12-\frac{(15-\dloc3)(\dloc3+7)}
 {2(3\dloc3+5)(\dloc3+3)}<12.
\]
Finally, if \(\dloc4=4\), all three multiplicities in
\eqref{eq:selected-classes} are at least two.
Each \(M/C_{xy}\), with \(xy\in E(G_B)\), has at least
\(\dloc3+5\) nonloops. Moreover,
\(X_B=B\), so \(\sum_{x\in X_B}\delta_B(x)=2\dloc3\).
By \autoref{lem:rank-two}\(\mathrm{(iii)}\) and
\autoref{lem:local}\(\mathrm{(iii)}\),
\[
 \wloc3\le\frac{\dloc3}{4\dloc3+8},\qquad
 \wloc4\le1-\frac{\dloc3}{2(\dloc3+4)}.
\]
Therefore, the left-hand side of \eqref{eq:rank-four-local-target} is at most
\[
 9+\frac{5\dloc3}{\dloc3+2}-\frac{3\dloc3}{2(\dloc3+4)}
 =12-\frac{(8-\dloc3)(\dloc3+6)}
 {2(\dloc3+2)(\dloc3+4)}<12.
\]
Thus, each summand in \eqref{eq:rank-four-target} is less than \(1/5\),
and \(d(M)<2b(M)/5\).
\end{proof}

For rank at least five, the deletion induction in
\autoref{sec:proofs} produces the correction term
\(\sum_{j=3}^{r(M)+1}(j-r(M))d_j(M)\).
Since \(j-r(M)\le-2\) for \(j\le r(M)-2\),
\begin{equation*}
 \sum_{j=3}^{r(M)+1}(j-r(M))d_j(M)+2d(M)
 \le d_{r(M)-1}(M)+2d_{r(M)}(M)+3d_{r(M)+1}(M).
\end{equation*}
Thus, after the common contribution \(-2d(M)\) is removed,
only the three largest circuit sizes require an upper bound.
Dividing these three coefficients by two gives the combination in the
next proposition.
\begin{proposition}\label{prop:joint}
Let \(M\) be a simple cosimple binary matroid with \(r(M)\ge5\). Then
\[
 \frac12d_{r(M)-1}(M)+d_{r(M)}(M)+\frac32d_{r(M)+1}(M)
 <\frac{2b(M)}{r(M)+1}.
\]
\end{proposition}
\begin{proof}
\emph{Rank five.}
Multiplying the inequality in \autoref{lem:cube-count}\(\mathrm{(ii)}\) by \(3/2\)
and adding \(3d_4(M)\) matches the circuit coefficients in the statement.
Expressing \(3d_4(M)\) through \eqref{eq:weight-total} then yields
\begin{equation}
 6\bigl(\frac12d_4(M)+d_5(M)+\frac32d_6(M)\bigr)\le\frac32b(M)+\sum_{B\in\BB(M)}
 \bigl(\frac34\wloc4+\frac3{40}\dloc5\bigr).
 \label{eq:rank-five-target}
\end{equation}
To bound the right-hand side of \eqref{eq:rank-five-target} strictly
below \(2b(M)\), it suffices to show that each summand is less than \(1/2\).
By \eqref{eq:graph-counts}, \(0\le\dloc4\le10\) and \(0\le\dloc5\le5\).
If \(\dloc4=0\), the summand is at most \(3/8\).
Suppose \(\dloc4>0\).
For \(\dloc5\le3\), we estimate the summand by \autoref{lem:local}\(\mathrm{(ii)}\):
\[
 \frac34\wloc4+\frac3{40}\dloc5
 \le\max\Bigl\{\frac14+\frac3{40}\dloc5,
 \frac7{16}-\frac3{160}\dloc5\Bigr\}
 \le\frac{19}{40}<\frac12.
\]
For \(\dloc5=4\), every \(M/C_{xy}\), with \(xy\in E(G_B)\),
has three nonempty parallel classes and at least \(\dloc4+5\) nonloops
by \eqref{eq:selected-classes}.
Since \(\dloc4\le10\), every four-element fundamental circuit \(C\)
satisfies the following bound by \autoref{lem:rank-two}\(\mathrm{(ii)}\):
\[
 b(M/C)\ge3\dloc4+8\ge\frac{15}4\dloc4+\frac12>\frac{15}4\dloc4.
\]
Thus, \(\wloc4<4/15\), and
\(3\wloc4/4+3\dloc5/40<1/2\).
Finally, if \(\dloc5=5\), the required bound reduces to \(\wloc4<1/6\).
Fix \(xy\in E(G_B)\), and let \(a,b,c\)
count the selected elements in the three parallel classes of
\(M/C_{xy}\), in the order displayed in \eqref{eq:selected-classes}.
Here \(X_B=B\), so the counts in \eqref{eq:selected-classes} satisfy
\(a+b+c=\dloc4+6\), \(2\le a,b\le5\), and
\(3\le c=3+|E(G_B[B-\{x,y\}])|\le6\).
Hence, \(abc>0\) and \((6-a)(6-b)(6-c)\ge0\).
Pairs of selected elements from different parallel classes form distinct bases.
Since \(a+b+c=\dloc4+6\), expanding
\((6-a)(6-b)(6-c)\) gives
\[
 b(M/C_{xy})\ge ab+ac+bc
 =6\dloc4+\frac{abc+(6-a)(6-b)(6-c)}6
 >6\dloc4.
\]
It follows that \(\wloc4<1/6\), and
\(3\wloc4/4+3\dloc5/40<1/2\).
Summing these strict bounds in \eqref{eq:rank-five-target}
establishes the rank-five case.

\emph{Rank at least six.}
To express the required circuit count through local weights,
multiply the bound in \autoref{lem:cube-count}\(\mathrm{(iii)}\) by \(3/2\).
Add \(\frac{r(M)+1}2d_{r(M)-1}(M)+\frac{8-r(M)}2d_{r(M)}(M)\)
to both sides and use \eqref{eq:weight-total} to obtain
\begin{align}
 &(r(M)+1)\Bigl(\frac12d_{r(M)-1}(M)+d_{r(M)}(M)+\frac32d_{r(M)+1}(M)\Bigr)\le\frac32b(M)+\notag\\
 &\sum_{B\in\BB(M)}\biggl(\frac{r(M)+1}{2(r(M)-1)}\wloc{r(M)-1}+\frac{8-r(M)}{2r(M)}\wloc{r(M)}+\frac3{16r(M)}\dloc{r(M)}\biggr).
 \label{eq:high-rank-target}
\end{align}
Since \((r(M)+1)/(2(r(M)-1))\le3/4\) and
\((8-r(M))/(2r(M))\le1/6\), the nonnegativity of the weights
shows that it suffices to prove
\begin{equation}\label{eq:high-rank-local-target}
 \frac34\wloc{r(M)-1}+\frac16\wloc{r(M)}
 +\frac3{16r(M)}\dloc{r(M)}<\frac12.
\end{equation}
If \(\dloc{r(M)}\le1\), \autoref{lem:local}\(\mathrm{(i),(ii)}\)
and \(r(M)\ge6\)
bound the left-hand side of
\eqref{eq:high-rank-local-target} by
\[
 \frac34\Bigl(\frac7{12}-\frac{\dloc{r(M)}}8\Bigr)
 +\frac{\dloc{r(M)}}{18}+\frac{\dloc{r(M)}}{32}
 =\frac7{16}-\frac{\dloc{r(M)}}{144}<\frac12.
\]
In the remaining case \(2\le\dloc{r(M)}\le r(M)\),
\autoref{lem:local}\(\mathrm{(ii)}\) implies
\(\wloc{r(M)-1}\le1/3\).
To obtain the two weight coefficients in
\eqref{eq:high-rank-local-target}, multiply this bound by \(1/4\)
and \autoref{lem:local}\(\mathrm{(iv)}\) by \(1/6\), then add.
Since \(r(M)\ge6\), the left-hand side of
\eqref{eq:high-rank-local-target} is at most
\begin{align*}
 &\frac14+\frac1{3\dloc{r(M)}}+\frac{3\dloc{r(M)}}{16r(M)}\\
 &=\frac7{16}+\frac1{3r(M)}
 -\frac{r(M)-\dloc{r(M)}}{r(M)}
  \Bigl(\frac3{16}-\frac1{3\dloc{r(M)}}\Bigr)\\
 &\le\frac7{16}+\frac1{3r(M)}<\frac12.
\end{align*}
Summing these estimates in \eqref{eq:high-rank-target} bounds its
right-hand side strictly below \(2b(M)\), as required.
\end{proof}

\section{Proofs of the main results}\label{sec:proofs}

The rank and deletion bounds are proved by induction on \(|E(M)|\).
For connected matroids, the joint bounds in \autoref{sec:weights}
handle the cosimple case of the rank induction, while series-pair
reductions handle the remaining case. The deletion induction also
uses the rank bound and a basis-exchange estimate.
Finally, summing the series-extension increments and applying the
deletion bound gives the growth corollary.

\subsection{Series-pair reductions}

We record the circuit descriptions and counting identities for
the two constructions used in the series-pair reductions.

A \emph{one-element series extension} \(M\) of \(N\) at a noncoloop
\(f\) replaces \(f\) by a series pair \(\{x,y\}\). Contracting \(x\)
and relabelling \(y\) as \(f\) recovers \(N\). The circuit set of \(M\) is
\begin{equation}\label{eq:series-circuits}
 \CC(M)=\CC(N\backslash f)\cup
 \{(C-f)\cup\{x,y\}:f\in C\in\CC(N)\}.
\end{equation}
For matroids \(N_1,N_2\) with ground sets meeting only at an
element \(f\) that is neither a loop nor a coloop of either,
their \emph{parallel connection} has circuit set
\begin{equation}\label{eq:parallel-circuits}
 \CC(P(N_1,N_2))\coloneqq\CC(N_1)\cup\CC(N_2)\cup
 \{(C_1\cup C_2)-f:f\in C_i\in\CC(N_i),\ i=1,2\}.
\end{equation}
In particular, a parallel connection of two circuits \(C_1,C_2\),
with \(C_1\cap C_2=\{f\}\), is binary and has exactly the three circuits
\(C_1\), \(C_2\), and \((C_1\cup C_2)-f\).
\begin{lemma}[{\cite[proof of Theorem~3.1, pp.~177--178]{Oxley1983}}]\label{lem:series}
Let \(M\) be a connected simple binary matroid of rank at least
three with a series pair. There are a connected simple binary minor
\(N\) with no coloops, where \(r(N)=r(M)-1\), and an element
\(f\in E(N)\), such that one of the following holds.
\begin{enumerate}
\item[\(\mathrm{(i)}\)] The matroid \(M\) is a one-element series
extension of \(N\) at \(f\), and
\[
 b(M)=b(N)+b(N\backslash f),\qquad d(M)=d(N).
\]
\item[\(\mathrm{(ii)}\)] The matroid \(M\) is the parallel connection
of \(N\) and a triangle \(\{f,x,y\}\), and
\[
 b(M)=2b(N)+b_f(N),\qquad d(M)=d(N)+d_f(N)+1.
\]
\end{enumerate}
\end{lemma}
The circuit counts in \autoref{lem:series} follow from
\eqref{eq:series-circuits} and \eqref{eq:parallel-circuits}, respectively.
For the basis counts in part \(\mathrm{(i)}\),
write \(\{x,y\}\) for the new series pair.
From a basis \(B\) of \(N\), form \((B-f)\cup\{x,y\}\) if \(f\in B\),
and \(B\cup\{x\}\), \(B\cup\{y\}\) otherwise.
In part \(\mathrm{(ii)}\), the corresponding bases are
\(B\cup\{x\}\), \(B\cup\{y\}\),
and, when \(f\in B\), \((B-f)\cup\{x,y\}\).
In each case, these are all the bases of \(M\), and each occurs once.
The basis counts follow from \(b(N)-b_f(N)=b(N\backslash f)\).
The series-extension identities in part \(\mathrm{(i)}\) also hold
without a connectedness assumption.

For the later case distinctions, suppose that a matroid \(N\) has no
coloops and \(d(N\backslash f)=0\) for some \(f\in E(N)\).
The set \(E(N)-f\) is independent.
Since \(f\) is not a coloop, \(E(N)-f\) is a basis of \(N\),
so \(C(f,E(N)-f)\) is the only circuit of \(N\).
Every element outside this circuit is a coloop. Therefore,
\begin{equation}\label{eq:circuit-case}
 d(N\backslash f)=0\quad\Longrightarrow\quad E(N)\in\CC(N).
\end{equation}

In case \(\mathrm{(ii)}\) of \autoref{lem:series}, we also need a lower
bound for \(b_f(N)\) in terms of \(d_f(N)\).
The following estimate is implicit in the
basis-extension argument in \cite[p.~171, proof of Theorem~1.3]{Oxley1983}.
\begin{lemma}\label{lem:pointed}
If \(M\) is a simple binary matroid and \(e\in E(M)\), then
\[
 b_e(M)\ge d_e(M)+1.
\]
\end{lemma}
\begin{proof}
Fix a basis \(B_e\) containing \(e\). We construct an injection
from the circuits containing \(e\) to the bases containing \(e\)
other than \(B_e\).
The fundamental cocircuit satisfies
\(C^*(e,B_e)\cap B_e=\{e\}\).
For each circuit \(C\) containing \(e\), we choose the omitted
element \(x\) so that the resulting basis differs from \(B_e\).
If \(C\cap B_e\ne\{e\}\), choose
\(x\in(C\cap B_e)-e\). If \(C\cap B_e=\{e\}\),
\eqref{eq:binary-orthogonality} makes
\((C\cap C^*(e,B_e))-e\) nonempty, and we choose \(x\) in this set.
The set \(C-x\) is independent and spans \(C\).
Hence, \((C-x)\cup(B_e-C)\) spans \(M\). Extend \(C-x\)
within this set to a basis \(B_C\), fixing one such choice for each
\(C\). Then \(e\in B_C\) and \(x\notin B_C\).
If \(C\cap B_e\ne\{e\}\), then \(x\in B_e-B_C\), so \(B_C\ne B_e\).
If \(C\cap B_e=\{e\}\), then \(B_C\) contains the
\(|C|-2\ge1\) elements of \(C-\{e,x\}\) outside \(B_e\),
so again \(B_C\ne B_e\).

Define
\[
 \varphi:\{C\in\CC(M):e\in C\}\to
 \{B\in\BB(M):e\in B,\ B\ne B_e\}
\]
by \(\varphi(C)=B_C\). To prove injectivity, we first
distinguish the two cases by parity. Every element added when extending
\(C-x\) lies in \(B_e-C\), which is disjoint from \(C^*(e,B_e)\).
Thus, \(B_C\cap C^*(e,B_e)=(C-x)\cap C^*(e,B_e)\).
If \(C\cap B_e\ne\{e\}\), then \(x\notin C^*(e,B_e)\).
Otherwise, \(x\in C^*(e,B_e)\).
By \eqref{eq:binary-orthogonality},
\[
 |B_C\cap C^*(e,B_e)|\equiv
 \begin{cases}
 0,&C\cap B_e\ne\{e\},\\
 1,&C\cap B_e=\{e\}
 \end{cases}
  \pmod2.
\]
Let \(C_1,C_2\) be circuits containing \(e\), and suppose
\(B_{C_1}=B_{C_2}\). By the preceding parity relation,
\(C_1\cap B_e=\{e\}\) if and only if \(C_2\cap B_e=\{e\}\).
If \(C_1\cap B_e\ne\{e\}\), then
\(C_1-B_e=B_{C_1}-B_e=B_{C_2}-B_e=C_2-B_e\), so
\(C_1\triangle C_2\subseteq B_e\).
If \(C_1\cap B_e=\{e\}\), then
\(C_i=(B_{C_i}-B_e)\cup\{e,x_i\}\) for \(i=1,2\), where \(x_i\)
is the element chosen for \(C_i\). Hence,
\(C_1\triangle C_2\subseteq\{x_1,x_2\}\).
Since \(B_e\) is a basis and \(M\) is simple,
\(C_1\triangle C_2\) is independent in both cases.
Thus, \eqref{eq:binary-symmetric-difference} implies
\(C_1\triangle C_2=\varnothing\), so \(\varphi\) is injective.
Consequently, \(d_e(M)\le b_e(M)-1\).
\end{proof}

\subsection{Proof of  \autoref{thm:rank}}

\begin{proof}[Proof of \autoref{thm:rank}]
We use strong induction on \(|E(M)|\), including the equality classification.
The hypotheses imply \(r(M)\ge2\) and \(d(M)\ge1\).
For \(r(M)\le3\), \eqref{eq:oxley} gives
\(2b(M)\ge4d(M)\ge(r(M)+1)d(M)\), with equality if and only if
\(M\cong F_7\). It remains to prove strict inequality for \(r(M)\ge4\).

\emph{Disconnected matroids.}
Write \(M=M_1\oplus M_2\) as a direct sum of nonempty matroids.
Each summand has no coloops, so \(r(M_i)\ge2\).
By \eqref{eq:exchange} and the absence of coloops, every element of a
fixed basis can be exchanged for an element outside it, producing
\(r(M_i)\) further distinct bases.
Thus, \(b(M_i)\ge r(M_i)+1\).
If both ranks are at least three, the induction hypothesis,
\eqref{eq:direct-sum-counts}, and \(b(M_i)\ge r(M_i)+1\) give
\[
 \frac{d(M)}{b(M)}
 =\frac{d(M_1)+d(M_2)}{b(M_1)b(M_2)}
 \le\frac4{(r(M_1)+1)(r(M_2)+1)}
 <\frac2{r(M)+1}.
\]
If one summand has rank two, relabel the summands so that
\(r(M_1)=2\). Then \(M_1\cong U_{2,3}\), the rank-two uniform
matroid on three elements.
Adding \(2b(M_2)\ge(r(M_2)+1)d(M_2)\) to
\(b(M_2)\ge r(M_2)+1\) yields
\[
 \frac{d(M)}{b(M)}
 =\frac{1+d(M_2)}{3b(M_2)}
 \le\frac1{r(M_2)+1}
 <\frac2{r(M_2)+3}=\frac2{r(M)+1}.
\]

\emph{Connected cosimple matroids.}
Rank four is covered by \autoref{prop:rank-four}.
Suppose \(r(M)\ge5\).
Since \(M\) is cosimple, every deletion has rank \(r(M)\)
and no coloops.
Consequently, the induction hypothesis applies to every deletion:
\[
 2b(M\backslash e)\ge(r(M)+1)d(M\backslash e),\qquad e\in E(M).
\]
Since each basis occurs in exactly \(|E(M)|-r(M)\) single-element
deletions and each \(j\)-element circuit in exactly \(|E(M)|-j\), we have
\[
 \sum_{e\in E(M)}b(M\backslash e)=(|E(M)|-r(M))b(M),\qquad
 \sum_{e\in E(M)}d(M\backslash e)
 =\sum_{j=3}^{r(M)+1}(|E(M)|-j)d_j(M).
\]
Summing the induction inequalities over \(e\in E(M)\) and applying
these two identities isolates the correction term:
\begin{align*}
 0&\le\sum_{e\in E(M)}
 \bigl(2b(M\backslash e)-(r(M)+1)d(M\backslash e)\bigr)\\
 &=(|E(M)|-r(M))\bigl(2b(M)-(r(M)+1)d(M)\bigr)+(r(M)+1)\sum_{j=3}^{r(M)+1}(j-r(M))d_j(M).
\end{align*}
To close the induction, we bound the last sum by a multiple of
\(2b(M)-(r(M)+1)d(M)\).
For \(3\le j\le r(M)-2\), the coefficient \(j-r(M)\) is at most
\(-2\). Separating these sizes and applying \autoref{prop:joint} gives
\begin{align*}
 \sum_{j=3}^{r(M)+1}(j-r(M))d_j(M)
 &\le -2\sum_{j=3}^{r(M)-2}d_j(M)-d_{r(M)-1}(M)+d_{r(M)+1}(M)\\
 &=2\Bigl(\frac12d_{r(M)-1}(M)+d_{r(M)}(M)
          +\frac32d_{r(M)+1}(M)-d(M)\Bigr)\\
 &<\frac{4b(M)}{r(M)+1}-2d(M).
\end{align*}
Inserting the bound on \(\sum_{j=3}^{r(M)+1}(j-r(M))d_j(M)\)
into the deletion-sum identity bounds its right-hand side strictly:
\[
 0\le\sum_{e\in E(M)}\bigl(2b(M\backslash e)-(r(M)+1)d(M\backslash e)\bigr)<(|E(M)|-r(M)+2)\bigl(2b(M)-(r(M)+1)d(M)\bigr).
\]
Since \(|E(M)|-r(M)+2>0\), the desired strict bound follows.

\emph{Connected matroids with a series pair.}
Since \(M\) has no coloops, failure of cosimplicity means that
\(M\) has a series pair.
Choose \(N,f\) as in \autoref{lem:series}.
The minor \(N\) has fewer elements than \(M\), so the induction
hypothesis applies.
In case \(\mathrm{(i)}\) of \autoref{lem:series}, \(M\) is a
one-element series extension of \(N\) at \(f\).
We first bound \(b(N\backslash f)\) by counting circuits containing \(f\).
For each circuit \(C\) containing \(f\), fix a basis \(B_C\)
of \(N\backslash f\) extending \(C-f\).
Since \(f\) is not a coloop, \(B_C\) is also a basis of \(N\),
and \(C=C(f,B_C)\). Thus, \(C\mapsto B_C\) is injective.
The injection and \eqref{eq:oxley} applied to \(N\backslash f\) imply
\[
 b(N\backslash f)\ge d_f(N),\qquad
 b(N\backslash f)\ge2d(N\backslash f).
\]
Adding twice the first bound to the second eliminates \(d_f(N)\):
\[
 3b(N\backslash f)\ge2d_f(N)+2d(N\backslash f)=2d(N).
\]
The induction hypothesis and \autoref{lem:series}\(\mathrm{(i)}\) now imply
\[
 2b(M)=2b(N)+2b(N\backslash f)
 \ge\Bigl(r(N)+\frac73\Bigr)d(N)
 >(r(M)+1)d(M).
\]

In case \(\mathrm{(ii)}\) of \autoref{lem:series},
\(M=P(N,M|\{f,x,y\})\). Suppose first that
\(d(N\backslash f)\ge1\).
Substituting the two identities in \autoref{lem:series}\(\mathrm{(ii)}\)
and applying \autoref{lem:pointed}, with
\(d_f(N)=d(N)-d(N\backslash f)\), yields
\begin{align*}
 2b(M)-(r(M)+1)d(M)
 &=4b(N)+2b_f(N)-(r(N)+2)(d(N)+d_f(N)+1)\\
 &\ge2\bigl(2b(N)-(r(N)+1)d(N)\bigr)
       +r(N)\bigl(d(N\backslash f)-1\bigr).
\end{align*}
Both terms on the right-hand side are nonnegative.
By induction, the first is positive unless \(N\cong F_7\).
In that exceptional case, \eqref{eq:fano} gives
\(r(N)(d(N\backslash f)-1)=3(7-1)=18>0\).
Therefore, \(2b(M)>(r(M)+1)d(M)\).
If instead \(d(N\backslash f)=0\), then \(E(N)\) is a circuit by
\eqref{eq:circuit-case}. In this case,
\(b(M)=3r(N)+2\) and \(d(M)=3\) by
\autoref{lem:series}\(\mathrm{(ii)}\), so
\[
 2b(M)-(r(M)+1)d(M)=3r(N)-2>0.
\]
The bound is strict in every case with \(r(M)\ge4\),
completing the induction.
\end{proof}

\subsection{Proofs of \autoref{thm:deletion} and  \autoref{cor:growth}}

The rank bound controls all bases, whereas \autoref{thm:deletion}
requires a bound for bases avoiding \(e\).
The following exchange estimate bounds the bases containing \(e\)
in terms of those avoiding \(e\), allowing us to pass from the rank
bound to the deletion bound in the cosimple case.

\begin{lemma}\label{lem:exchange-count}
Let \(M\) be a simple cosimple binary matroid with \(r(M)\ge3\).
For every \(e\in E(M)\),
\[
 2b_e(M)\le(r(M)-1)b(M\backslash e).
\]
\end{lemma}
\begin{proof}
We count basis exchanges that remove \(e\). By \eqref{eq:exchange}
and cosimplicity, each basis \(B\) containing \(e\) admits
\(|C^*(e,B)|-1\ge2\) such exchanges. In the reverse direction,
a basis \(B\) avoiding \(e\) admits \(|C(e,B)|-1\) exchanges that
insert \(e\), at most \(r(M)-1\) unless \(C(e,B)\) is spanning.
We compensate for each spanning-circuit exception by an extra exchange
from a distinct basis containing \(e\).
Let \(q\) be the number of spanning circuits containing \(e\).

To obtain the extra exchange, we first show that each such circuit
\(C\) has an element \(x\in C-e\) satisfying
\(|C^*(e,C-x)|\ge4\).
Put \(B=C-e\). If no such \(x\) exists, cosimplicity makes
\(C^*(e,C-x)\) a triad for every \(x\in B\).
Since \(C^*(e,C-x)\) avoids \(C-\{e,x\}\),
\eqref{eq:binary-orthogonality} forces
\(C^*(e,C-x)\cap C=\{e,x\}\).
The cocircuit \(C^*(e,C-x)\) meets \(B\) only in \(x\),
so uniqueness of the fundamental cocircuit gives
\[
 C^*(e,C-x)=C^*(x,B)=\{e,x,z_x\},\qquad z_x\notin C.
\]
Fix \(x\in B\). By \eqref{eq:exchange}, \(x\in C(z_x,B)\).
Since \(M\) is simple, \(C(z_x,B)\) contains an element
\(y\notin\{x,z_x\}\).
By a second application of \eqref{eq:exchange},
\(z_x\in C^*(y,B)=\{e,y,z_y\}\).
As \(z_x\notin C\), we have \(z_x=z_y\).
The dual of \eqref{eq:binary-symmetric-difference} now implies that
\(C^*(x,B)\triangle C^*(y,B)=\{x,y\}\) contains a cocircuit,
contrary to cosimplicity.

For each of the \(q\) spanning circuits \(C\), choose \(x\in C-e\)
with \(|C^*(e,C-x)|\ge4\).
The bases \(C-x\) are distinct by \eqref{eq:spanning-bases}, and each admits
at least three exchanges removing \(e\). Every other basis containing
\(e\) admits at least two.
Since \(e\) is not a coloop, the bases avoiding \(e\) are exactly
the bases of \(M\backslash e\).
Among them, precisely the \(q\) bases \(C-e\) have a spanning
fundamental circuit containing \(e\). Each contributes \(r(M)\)
exchanges, while every other basis contributes at most \(r(M)-1\).
Double-counting the exchanges that remove \(e\) by \eqref{eq:exchange} yields
\begin{align*}
 2b_e(M)+q
 &\le\sum_{B\in\BB(M),\,e\in B} (|C^*(e,B)|-1)\\
 &=\sum_{B\in\BB(M),\,e\notin B} (|C(e,B)|-1)\\
 &\le(r(M)-1)b(M\backslash e)+q.
\end{align*}
Subtracting \(q\) proves the assertion.
\end{proof}

\begin{proof}[Proof of \autoref{thm:deletion}]
We use strong induction on \(|E(M)|\), including the equality classification.
If \(E(M)\) is a circuit, then \(b(M\backslash e)=d(M)=1\).
Every rank-two matroid under consideration is \(U_{2,3}\),
so we may assume that \(E(M)\) is not a circuit and \(r(M)\ge3\).

\emph{Disconnected matroids.}
Write \(M=M_1\oplus M_2\), where both summands are nonempty
and \(e\in E(M_1)\).
Neither summand has coloops, so \(d(M_1),d(M_2)\ge1\).
By the induction hypothesis, \eqref{eq:direct-sum-counts}, and
\eqref{eq:oxley},
\[
 b(M\backslash e)=b(M_1\backslash e)b(M_2)\ge d(M_1)b(M_2)\ge2d(M_1)d(M_2)\ge d(M_1)+d(M_2)=d(M).
\]
Equality in \(2d(M_1)d(M_2)\ge d(M_1)+d(M_2)\) forces
\(d(M_1)=d(M_2)=1\),
and the absence of coloops then makes \(E(M_2)\) its unique circuit.
Thus, \(b(M_2)=|E(M_2)|\ge3>2d(M_2)\), making the second
inequality in the displayed chain strict.

\emph{Connected cosimple matroids.}
For \(M\not\cong F_7\), \autoref{thm:rank} and
\autoref{lem:exchange-count} imply
\[
 d(M)<\frac{2b(M)}{r(M)+1}
 =\frac{2(b_e(M)+b(M\backslash e))}{r(M)+1}
 \le b(M\backslash e).
\]
For \(F_7\), the deletion bound is strict by \eqref{eq:fano}.

\emph{Connected matroids with a series pair.}
The preceding cases are strict, so only connected matroids with a
series pair can attain equality.
Choose \(N,f\) as in \autoref{lem:series}.
Since \(N\) has fewer elements than \(M\), the induction hypothesis
applies to \(N\).
In case \(\mathrm{(i)}\) of \autoref{lem:series}, \(M\) is a
one-element series extension of \(N\) at \(f\), replacing
\(f\) by \(\{x,y\}\).
Restricting the basis description following
\autoref{lem:series}\(\mathrm{(i)}\) to bases avoiding \(e\) yields
\begin{equation}\label{eq:series-deletions}
 b(M\backslash e)=
 \begin{cases}
 b(N\backslash f),&e\in\{x,y\},\\
 2b(N\backslash e)-b_f(N\backslash e),&e\in E(N)-f.
 \end{cases}
\end{equation}
By induction, the first value in \eqref{eq:series-deletions}
is at least \(d(N)=d(M)\).
For \(e\in E(N)-f\), the induction hypothesis and
\(b_f(N\backslash e)\le b(N\backslash e)\) give
\begin{equation}\label{eq:series-deletion-deficit}
 \begin{aligned}
 b(M\backslash e)-d(M)
 &=\bigl(b(N\backslash e)-d(N)\bigr)\\
 &\quad+\bigl(b(N\backslash e)-b_f(N\backslash e)\bigr)\ge0.
 \end{aligned}
\end{equation}

For the equality case, suppose \(b(M\backslash e)=d(M)\).
Since \(E(M)\) is not a circuit, neither is \(E(N)\), by
\eqref{eq:series-circuits}.
If \(e\in\{x,y\}\), equality in \eqref{eq:series-deletions}
is equivalent to \(b(N\backslash f)=d(N)\). By induction,
\(N=P(N|C_1,N|C_2)\), where \(C_1\) is a circuit, \(C_2\) is a
triangle, and \(f\notin C_2\).
If \(e\in E(N)-f\), equality forces both nonnegative terms in
\eqref{eq:series-deletion-deficit} to vanish. By induction,
\(N=P(N|C_1,N|C_2)\),
where \(C_1\) is a circuit, \(C_2\) is a triangle, and \(e\notin C_2\).
The equality \(b_f(N\backslash e)=b(N\backslash e)\) means that
every basis of \(N\backslash e\) contains \(f\), making \(f\) a coloop.
Since \(e\notin C_2\), the triangle \(C_2\) remains a circuit of
\(N\backslash e\). Therefore, \(f\notin C_2\).
In both cases, \eqref{eq:series-circuits} shows that \(M\) is the
parallel connection of \((C_1-f)\cup\{x,y\}\) and \(C_2\), with
\(e\notin C_2\).

Next consider case \(\mathrm{(ii)}\) of \autoref{lem:series}, where
\(M=P(N,M|\{f,x,y\})\).
To isolate the equality cases, we distinguish whether
\(N\backslash f\) has a circuit. The circuit count in that lemma becomes
\begin{equation}\label{eq:triangle-circuits}
 d(M)=2d(N)-d(N\backslash f)+1.
\end{equation}
From the basis description following \autoref{lem:series}\(\mathrm{(ii)}\),
we compute
\begin{equation}\label{eq:triangle-deletions}
 b(M\backslash e)=
 \begin{cases}
 b(N),&e\in\{x,y\},\\
 b(N)+b(N\backslash f),&e=f,\\
 2b(N\backslash e)+b_f(N\backslash e),&e\in E(N)-f.
 \end{cases}
\end{equation}
If \(d(N\backslash f)\ge1\), then \eqref{eq:oxley} and
\eqref{eq:triangle-circuits} imply
\[
 b(N)\ge2d(N)\ge d(M).
\]
Equality \(b(N)=d(M)\) requires \(N\cong F_7\) and
\(d(N\backslash f)=1\), contrary to \eqref{eq:fano}.
Thus, \(b(M\backslash e)>d(M)\) for \(e\in\{x,y,f\}\)
by \eqref{eq:triangle-deletions}.
For \(e\in E(N)-f\), the element \(f\) is a nonloop of
\(N\backslash e\), so \(b_f(N\backslash e)\ge1\).
The induction hypothesis and \eqref{eq:triangle-deletions}
imply \(b(M\backslash e)\ge2d(N)+1>d(M)\).

It remains to consider \(d(N\backslash f)=0\), when \(E(N)\) is a
circuit by \eqref{eq:circuit-case}. Then \(d(M)=3\), and \(b(M\backslash e)\) equals
\(|E(N)|\), \(|E(N)|+1\), or \(3\), according to the three cases in
\eqref{eq:triangle-deletions}.
Thus, equality holds for every \(e\in E(N)-f\),
and also for \(e\in\{x,y\}\) when \(N\) is a triangle.
When \(N\) is a triangle and \(e\in\{x,y\}\), interchanging the
two triangles places \(e\) outside the designated triangle, as required.

Conversely, equality holds when \(E(M)\) is a circuit. A parallel connection
of a circuit \(C_1\) and a triangle \(C_2\) has three circuits
by \eqref{eq:parallel-circuits}.
Taking \(N=M|C_1\) in \eqref{eq:triangle-deletions}, we count
exactly three bases after deleting any element outside \(C_2\).
This completes the equality classification and the induction.
\end{proof}

The growth bound follows by summing the basis increments of the
successive series extensions, each controlled by \autoref{thm:deletion}.

\begin{proof}[Proof of \autoref{cor:growth}]
For \(s=0\), we have \(M=N\) and equality holds. Assume \(s\ge1\).
Take successive matroids \(N_0,\ldots,N_s\), with \(N_0=N\) and \(N_s=M\),
where \(N_i\) is obtained from \(N_{i-1}\) by a one-element series
extension at \(f_i\) for \(1\le i\le s\).
Every \(N_i\) is simple and binary and has no coloops.
By the series-extension identities in \autoref{lem:series}\(\mathrm{(i)}\)
and \autoref{thm:deletion},
\begin{equation}\label{eq:growth-step}
 b(N_i)-b(N_{i-1})=b(N_{i-1}\backslash f_i)
 \ge d(N_{i-1})=d(N).
\end{equation}
Summing \eqref{eq:growth-step} over \(1\le i\le s\) gives
\[
 b(M)-b(N)=\sum_{i=1}^s\bigl(b(N_i)-b(N_{i-1})\bigr)
 \ge s d(N).
\]
If \(N\) is cosimple,
the strict cosimple case of \autoref{thm:deletion} makes
\eqref{eq:growth-step} strict for \(i=1\).

Since each basis increment is at least \(d(N)\), equality
\(b(M)=b(N)+s d(N)\) forces equality at every step:
\(b(N_{i-1}\backslash f_i)=d(N)\) for every \(1\le i\le s\).
In particular, \autoref{thm:deletion} applied to
\(N_{s-1}\) and \(f_s\) shows that \(E(N_{s-1})\) is a circuit, or that
\(N_{s-1}\) is a parallel connection of a circuit and a triangle with
\(f_s\) outside the triangle. The series extension \(M\) has the same
form by \eqref{eq:series-circuits}.

Conversely, if \(E(M)\) is a circuit, \eqref{eq:series-circuits} shows that
every \(E(N_i)\) is a circuit. Hence,
\(b(N_{i-1}\backslash f_i)=d(N)=1\) at every step.
For the parallel-connection case, we trace the extensions backwards
to verify equality at every step.
Suppose \(M=P(M|C_1,M|C_2)\), where \(C_1\) is a circuit
and \(C_2\) is a triangle meeting \(C_1\) at \(f\).
Let \(\{x,y\}\) be the series pair introduced in the last extension.
Contracting \(x\) and relabelling \(y\) as \(f_s\) recovers
the predecessor \(N_{s-1}\).
Since \(N_{s-1}\) is simple, \eqref{eq:series-circuits} shows that
every circuit of \(M\) containing \(x\) or \(y\) has at least four
elements. Therefore, \(\{x,y\}\cap C_2=\varnothing\) and
\(\{x,y\}\subseteq C_1-f\).
Reversing \eqref{eq:series-circuits} and applying
\eqref{eq:parallel-circuits} identifies the contraction as
\(M/x=P((M|C_1)/x,M|C_2)\).
The circuit \(C_1-x\) has at least three elements because \(M/x\) is simple.
Hence, \(N_{s-1}\) is again a parallel connection of a circuit and \(C_2\),
with the extended element \(f_s\), corresponding to \(y\), outside \(C_2\).
Repeating the argument shows that every forward step extends an
element outside \(C_2\) and attains equality in \eqref{eq:growth-step}
by \autoref{thm:deletion}.
\end{proof}

\section{Distinct basis extensions}\label{sec:extensions}

In the weighted count \eqref{eq:weight-total}, each \(j\)-element
circuit contributes total weight \(j\).
We now ask when, for a prescribed finite set \(J\) of circuit sizes,
the pairs \((C,e)\), with \(C\in\CC(M)\), \(|C|\in J\), and
\(e\in C\), can be assigned distinct bases containing \(C-e\).
Such an assignment gives \(\sum_{j\in J}j d_j(M)\le b(M)\).

Each \(C-e\) extends to a basis, but the bases must be chosen
simultaneously, since different pairs may have common extensions.
We distribute weight one from each pair \((C,e)\) equally among
its \(b(M/C)\) extensions.
If each basis receives total weight at most one, every family of
pairs has at least as many available bases as pairs, as required
by Hall's theorem. We first establish this bound at a fixed basis.

\begin{lemma}\label{lem:incomparable-weight}
Let \(M\) be a simple binary matroid, \(B\in\BB(M)\), and
\(e_1,\ldots,e_k\) be distinct elements of \(E(M)-B\).
If the sets \(C(e_i,B)-e_i\) are pairwise incomparable under inclusion,
then
\[
 \sum_{i=1}^k\frac1{b(M/C(e_i,B))}\le1.
\]
\end{lemma}
\begin{proof}
If \(k=0\), the sum is zero. Suppose \(k\ge1\), and fix \(i\).
We construct \(k\) distinct bases of \(M/C(e_i,B)\).
For each \(j\ne i\), incomparability gives an element
\(x_j\in(C(e_j,B)-e_j)-C(e_i,B)\).
Since \(x_j\notin C(e_i,B)\), exchanging \(x_j\) for \(e_j\)
preserves \(C(e_i,B)-e_i\). The set \((B-x_j)\cup\{e_j\}\)
is a basis by \eqref{eq:exchange}. Thus, \eqref{eq:contraction-bases}
identifies the following bases
\[
 B-C(e_i,B),\qquad
 (B-C(e_i,B)-x_j)\cup\{e_j\},\qquad j\ne i
\]
of \(M/C(e_i,B)\). The displayed bases are distinct because their
intersections with \(E(M)-B\) are \(\varnothing\) and the distinct
singletons \(\{e_j\}\). Hence, \(b(M/C(e_i,B))\ge k\) for every \(i\),
and summing the reciprocals yields
\[
 \sum_{i=1}^k\frac1{b(M/C(e_i,B))}\le1.\qedhere
\]
\end{proof}

Hall's theorem \cite{Hall1935} states that a finite indexed family
of finite sets has a system of distinct representatives if and only
if the union of every subfamily has cardinality at least the number
of its members.
The circuit-size condition below ensures the incomparability
needed to apply \autoref{lem:incomparable-weight}.

\begin{theorem}\label{thm:distinct-extensions}
Let \(M\) be a simple binary matroid, and \(J\) be a finite set
of positive integers. Suppose that \(M\) has no
\((q-p+2)\)-element circuit whenever \(p,q\in J\) and \(p<q\).
Then there exists an injective map
\[
 \phi:\{(C,e):C\in\CC(M),\ |C|\in J,\ e\in C\}\to\BB(M)
\]
such that \(C-e\subseteq\phi(C,e)\) for every pair \((C,e)\).
In particular,
\[
 \sum_{j\in J}j d_j(M)\le b(M).
\]
\end{theorem}
\begin{proof}
Fix \(B\in\BB(M)\). To apply \autoref{lem:incomparable-weight}
at \(B\), we first show that the sets \(C(e,B)-e\),
where \(e\in E(M)-B\) and \(|C(e,B)|\in J\), are pairwise
incomparable under inclusion.
Suppose, to the contrary, that distinct such elements \(e,f\) satisfy
\(C(e,B)-e\subseteq C(f,B)-f\).
To express \(f\) in a basis containing \(e\), choose
\(x\in C(e,B)-e\) and put \(B'=(B-x)\cup\{e\}\),
a basis by \eqref{eq:exchange}.
The binary column relations \eqref{eq:circuit-column-sum} give
\[
 f=e+\sum_{y\in C(f,B)-(C(e,B)\cup\{f\})}y.
\]
Since \(x\in C(e,B)\), all vectors on the right belong to \(B'\).
The unique expansion in the basis \(B'\) determines
\(C(f,B')=C(e,B)\triangle C(f,B)\).
Since every circuit of \(M\) has at least three elements,
\[
 3\le |C(f,B')|=|C(f,B)|-|C(e,B)|+2.
\]
Hence, \(|C(e,B)|<|C(f,B)|\), and \(C(f,B')\) has a size excluded
by the hypothesis, a contradiction.

To verify Hall's condition, let \(\mathcal A\) be any set of pairs
\((C,e)\) as in the statement, and set
\[
 \BB_{\mathcal A}=
 \bigl\{B\in\BB(M):C-e\subseteq B
 \text{ for some }(C,e)\in\mathcal A\bigr\}.
\]
By \eqref{eq:contraction-bases}, summing \(1/b(M/C)\)
over all bases containing \(C-e\) gives \(1\) for each
\((C,e)\in\mathcal A\).
If \(C-e\subseteq B\), then \(e\notin B\) and \(C=C(e,B)\).
Therefore, distinct pairs contributing to a fixed basis have distinct
elements \(e\). By the incomparability proved above and
\autoref{lem:incomparable-weight}, each basis receives total weight
at most one. Counting the total weight by pairs and by bases gives
\[
 |\mathcal A|
 =\sum_{B\in\BB_{\mathcal A}}
   \sum_{(C,e)\in\mathcal A,\,C-e\subseteq B}
   \frac1{b(M/C)}
 \le|\BB_{\mathcal A}|.
\]
By Hall's theorem, there exist pairwise distinct bases \(B_{C,e}\)
satisfying \(C-e\subseteq B_{C,e}\), one for each pair \((C,e)\)
in the statement. Define \(\phi(C,e):=B_{C,e}\), which is the required injective map.
Since there are \(\sum_{j\in J}j d_j(M)\) such pairs, injectivity
gives the numerical bound.
\end{proof}

For a single circuit size \(j\), the numerical bound also follows
from Oxley's spanning-circuit bound recalled in \autoref{sec:weights}
by induction on \(|E(M)|\), allowing coloops.
The empty matroid has no circuits. Deleting a coloop preserves both
counts, so it remains to consider nonempty matroids with no coloops.
For \(j>r(M)+1\), we have \(d_j(M)=0\), while \(j=r(M)+1\)
is covered by the spanning-circuit bound.
For \(j\le r(M)\), summing the inductive bounds over all deletions gives
\[
 (|E(M)|-j)j d_j(M)=\sum_{e\in E(M)}j d_j(M\backslash e)\le\sum_{e\in E(M)}b(M\backslash e)=(|E(M)|-r(M))b(M).
\]
Since \(|E(M)|-j\ge|E(M)|-r(M)>0\), the desired bound follows.
In contrast, the joint bound can fail without the circuit-size exclusion:
\[
 3d_3(F_7)+4d_4(F_7)=49>28=b(F_7).
\]
For two circuit sizes, taking the contrapositive of the numerical
bound in \autoref{thm:distinct-extensions} gives the following corollary.

\begin{corollary}\label{cor:forced-size}
Let \(M\) be a simple binary matroid, and \(p<q\) be positive
integers. If \(p d_p(M)+q d_q(M)>b(M)\), then \(M\) has a
\((q-p+2)\)-element circuit.
\end{corollary}

The \emph{girth} of a matroid with at least one circuit is the
minimum size of a circuit.
The girth provides a sufficient condition for applying
\autoref{thm:distinct-extensions} to all circuit sizes in an interval.

\begin{corollary}\label{cor:girth-interval}
Let \(M\) be a simple binary matroid with at least one circuit
and girth \(g\). If \(m\le n\) are positive integers with
\(n-m\le g-3\), then
\[
 \sum_{j=m}^n j d_j(M)\le b(M).
\]
In particular, equality holds when \(E(M)\) is a circuit and
\(m=n=|E(M)|\).
\end{corollary}
\begin{proof}
For \(m\le p<q\le n\), we have
\[
 q-p+2\le n-m+2<g.
\]
Therefore, \autoref{thm:distinct-extensions} applies with
\(J=\{m,\ldots,n\}\) and gives the bound.
If \(E(M)\) is a circuit and \(m=n=|E(M)|\), then both
sides equal \(|E(M)|\).
\end{proof}

\section*{Acknowledgements}
This work is supported by the Guangdong Basic and Applied Basic Research Foundation (Grant No. 2026A1515012237).
\section*{Declaration of AI use}
The author developed the original results and proofs. OpenAI's Codex assisted with proofreading and language polishing. The author takes full responsibility for the content of this paper.

\end{document}